\documentclass[10pt]{article}

\usepackage[T1]{fontenc}

\usepackage{amsmath}
\usepackage{amsthm}
\usepackage{amssymb}

\usepackage{enumitem}

\usepackage{forest}
\usepackage{xcolor}
\definecolor{bluedp}{HTML}{63A3D1}
\definecolor{reddp}{HTML}{F57A7A}

\usepackage[small, pagestyles]{titlesec}
\usepackage[small, it]{caption}

\usepackage[square,comma,numbers,sort&compress]{natbib}

\theoremstyle{plain}
\newtheorem{theorem}{Theorem}[section]
\newtheorem{proposition}[theorem]{Proposition}

\newtheorem{corollary}[theorem]{Corollary}
\newtheorem{conjecture}[theorem]{Conjecture}
\newtheorem{observation}[theorem]{Observation}

\renewenvironment{abstract}
  {
    \begin{list}{}%
      {\setlength{\rightmargin}{1in}%
       \setlength{\leftmargin}{1in}}%
    \item[]\ignorespaces\begin{small}
  }
  {
    \end{small}\unskip\end{list}
  }

\newpagestyle{main}[\small]{
  \headrule
  \sethead[\usepage][][]
          {\sc Boolean--Eulerian numbers}{}{\usepage}
}

\titleformat{\section}
  {\large\sc}
  {\thesection.}{1em}{}

\newcommand{\boel}{\mathrm{BoEl}}
\DeclareMathOperator{\des}{des}
\DeclareMathOperator{\pk}{pk}

\usepackage[bookmarks]{hyperref}
\hypersetup{
	colorlinks=true,
	linkcolor=black,
	anchorcolor=black,
	citecolor=black,
	urlcolor=black,
	pdfpagemode=UseThumbs,
	pdftitle={Boolean--Eulerian numbers},
	pdfsubject={Combinatorics},
	pdfauthor={Bona and Vatter},
}

\title{\sc Boolean--Eulerian numbers}

\author{
	Mikl\'os B\'ona%
	\footnote{University of Florida, Gainesville, Florida, USA.}
	\quad
	and
	\quad
	Vincent Vatter%
	\footnotemark[1]
}

\date{\today}

\begin{document}
\maketitle

\forestset{
  my tree/.style={
    for tree={
      circle,
      draw,
      very thick,
      edge={very thick},
      minimum size=7.5mm,
      inner sep=0pt,
      s sep=8mm,
      l sep=5mm,
      fill=white,
      text=black,
      font=\bfseries\normalsize
    }
  }
}

\begin{abstract}
We study decreasing binary trees in which every vertex with two children is colored red or blue. We construct two bijections. The first, to ordered set partitions into odd-sized blocks each arranged as an alternating permutation, shows that the exponential generating function of these trees is $1/(1-\tan z)$. The second, to nonplane decreasing~$1$-$2$ trees paired with a binary label on each non-root vertex, proves combinatorially that the count equals $2^{n-1}$ times the~$n$th Euler number. Refining by the number of right edges yields the Boolean--Eulerian polynomials, which are an explicit algebraic transform of the classical Eulerian polynomials. The Foata--Strehl orbit decomposition, recast in the decreasing-binary-tree model, gives a direct combinatorial proof of gamma-positivity, and the algebraic transform carries real-rootedness and interlacing of zeros from the Eulerian polynomials to the Boolean--Eulerian polynomials.
\end{abstract}

\pagestyle{main}


\section{Introduction}
\label{sec:intro}

We write~$[n]$ for the set $\{1, 2, \ldots, n\}$ and $S_n$ for the set of permutations of~$[n]$. A \emph{descent} of a permutation $p = p_1 \cdots p_n$ is an index~$i$ with $1 \leq i \leq n-1$ and $p_i > p_{i+1}$; a \emph{peak} is an index~$i$ with $2 \leq i \leq n-1$ and $p_{i-1} < p_i > p_{i+1}$. We write $\des(p)$ and $\pk(p)$ for the number of descents and peaks of $p$. Euler has numbers attached to both of these statistics, and both appear in this paper, so we take a moment to distinguish them. The \emph{Eulerian numbers} count permutations of~$[n]$ by their number of descents, and arise as the coefficients of the Eulerian polynomial $A_n(x) = \sum_{p \in S_n} x^{\des(p)}$. The \emph{Euler number} $E_n$, by contrast, counts \emph{alternating} permutations, meaning permutations $p = p_1 \cdots p_n$ with $p_1 < p_2 > p_3 < \cdots$; equivalently, permutations of~$[n]$ in which every even index is a peak.

A \emph{decreasing binary tree} on~$[n]$ is a rooted tree on the vertex set~$[n]$ in which every vertex has a larger label than its children, every vertex has at most two children, and every child is designated left or right even when it is an only child. These trees are in bijection with permutations of~$[n]$: send $p = p_1 \cdots p_n$ to the tree $T(p)$ whose root is labeled~$n$, with the entries to the left and right of~$n$ in $p$ forming the left and right subtrees recursively, and recover $p$ by reading $T(p)$ in-order. Under this bijection, right edges of $T(p)$ correspond to descents of $p$, and vertices with two children correspond to peaks. See Figure~\ref{fig:perm-binary-tree} for an example.

\begin{figure}[t]
\centering
\begin{forest}
my tree
[7, fill=black!20,
  [6, fill=black!20,
    [1]
    [5
      [, no edge, draw=none]
      [3]
    ]
  ]
  [4
    [2]
    [, no edge, draw=none]
  ]
]
\end{forest}
\caption{The permutation \(p = 1653724\) and its decreasing binary tree \(T(p)\). The peaks of \(p\) are at positions $2$ and $5$, with values $6$ and $7$; these correspond to the shaded vertices of $T(p)$, namely the vertices with two children. The descents of \(p\) are at positions $2, 3, 5$, corresponding to the right edges of $T(p)$.}
\label{fig:perm-binary-tree}
\end{figure}
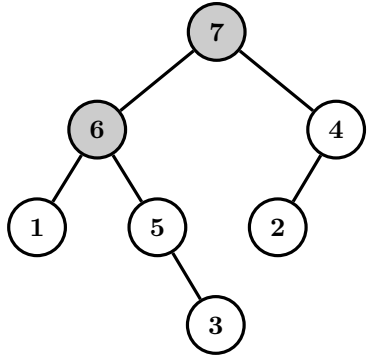

Binary plane trees on~$n$ unlabeled vertices are counted by the Catalan numbers, and refining by the number of right edges gives the Narayana numbers. In~\cite{bona:stack-sorting-p:, bona:boolean--naraya:}, the first author studied binary plane trees in which every vertex with two children receives an additional label from $\{0,1\}$, calling these \emph{$0$-$1$-trees}, and refined the count by the number of right edges. The resulting Boolean--Catalan and Boolean--Narayana numbers share many nice enumerative properties with the Catalan and Narayana numbers: symmetry, unimodality, log-concavity, real-rootedness, and interlacing. In this paper, we carry out the labeled analogue of this investigation.

A \emph{bicolored decreasing binary tree} on~$[n]$ is a decreasing binary tree on~$[n]$ in which every vertex with two children is additionally colored red or blue. If a permutation $p$ has $m$ peaks, then $T(p)$ has $m$ vertices with two children, each of which can be colored in two ways, so the total number of bicolored decreasing binary trees on~$[n]$ is
\[
	H_n = \sum_{p \in S_n} 2^{\pk(p)}.
\]
Refining by right edges, we define the \emph{Boolean--Eulerian number} $\boel(n,k)$ to be the number of bicolored decreasing binary trees on~$[n]$ with precisely $k-1$ right edges, and we write
\[
	B_n(u) = \sum_{k=1}^{n} \boel(n,k)\, u^{k-1} = \sum_{p \in S_n} 2^{\pk(p)}\, u^{\des(p)}
\]
for the \emph{Boolean--Eulerian polynomial}.

Our first results concern the total count $H_n$. In Section~\ref{sec:total}, we show that the exponential generating function of $H_n$ is $1/(1 - \tan z)$, and we construct a bijection between bicolored decreasing binary trees on~$[n]$ and ordered set partitions of~$[n]$ into blocks of odd size, each block arranged as an alternating permutation. In Section~\ref{sec:bijection}, we construct a second bijection, to pairs consisting of a nonplane decreasing~$1$-$2$ tree on~$[n]$ and a binary label on each non-root vertex. Since nonplane decreasing~$1$-$2$ trees on~$[n]$ are counted by the Euler number $E_n$, this proves combinatorially that $H_n = 2^{n-1} E_n$ for $n \geq 1$, and hence that
\[
	\frac{1}{1 - \tan z} = \frac{1 + \sec 2z + \tan 2z}{2}.
\]

The remaining sections concern the Boolean--Eulerian polynomials. In Section~\ref{sec:polynomials}, we derive a bivariate differential equation for their exponential generating function. The key structural result, in Section~\ref{sec:transform}, is that $B_n(u)$ is an explicit algebraic transform of the classical Eulerian polynomial $A_n(u)$: specifically, $B_n(u) = g(u)^{n-1}\, A_n(q(u))$ for explicit algebraic functions $g$ and $q$, where $q$ restricts to an increasing bijection from $(-\infty, 0)$ to itself. The known properties of the Eulerian polynomials therefore transfer to the Boolean--Eulerian polynomials; in particular, we show that $B_n(u)$ is gamma-positive, real-rooted, and that the family $\{B_n(u)\}_{n \geq 2}$ has interlacing zeros. For gamma-positivity we give a combinatorial proof, via a tree version of the Foata--Strehl orbit decomposition. Finally, in Section~\ref{sec:unimodal}, we give a combinatorial proof of unimodality by importing an injection from~\cite{bona:boolean--naraya:} that extends without modification to the labeled setting.


\section{Total enumeration}
\label{sec:total}

Write $H_n$ for the number of bicolored decreasing binary trees on~$[n]$, with $H_0 = 1$, and let $H(z) = \sum_{n \geq 0} H_n z^n/n!$ be the associated exponential generating function.

\begin{proposition}
\label{prp:diffeq}
The generating function $H(z)$ satisfies $H'(z) = H(z)^2 + (H(z) - 1)^2$, and therefore $H(z) = 1/(1 - \tan z)$.
\end{proposition}

\begin{proof}
The coefficient of $z^n/n!$ in $H'(z)$ is $H_{n+1}$. A bicolored decreasing binary tree on $[n+1]$ is determined by its root (labeled $n+1$), an ordered pair of bicolored decreasing binary trees on complementary subsets of~$[n]$ forming its left and right subtrees, and (if both subtrees are nonempty) a choice of red or blue for the root. Ordered pairs of bicolored decreasing binary trees are counted by $H(z)^2$, and pairs with both subtrees nonempty are counted by $(H(z)-1)^2$, so $H'(z) = H(z)^2 + (H(z)-1)^2$. The differential equation together with the initial condition $H(0) = 1$ determine the coefficients of $H(z)$ recursively, and one checks that $1/(1-\tan z)$ satisfies both.
\end{proof}

The generating function for the Euler numbers of odd index is $\sum_{n \text{ odd}} E_n z^n/n! = \tan z$, so the expansion $1/(1 - \tan z) = \sum_{m \geq 0} (\tan z)^m$ expresses $H(z)$ as the EGF for sequences of structures each counted by $\tan z$. The labeled sequence construction then gives the following combinatorial interpretation.

\begin{corollary}
\label{cor:sequences}
The number $H_n$ counts ordered set partitions of~$[n]$ into blocks of odd size, each block arranged as an alternating permutation.
\end{corollary}

We call these objects \emph{alternating block sequences} on~$[n]$, written $C_1 \mid C_2 \mid \cdots \mid C_m$ where each $C_i$ is an alternating permutation of odd length on a subset of~$[n]$. We now construct a bijection between bicolored decreasing binary trees on~$[n]$ and alternating block sequences on~$[n]$.

The construction relies on a merge-and-split property of alternating permutations of odd length. If~$\alpha$ is such a permutation of length at least $3$, then $\alpha$ begins with an ascent and ends with a descent, so its maximum lies at an even index. It follows that if $\alpha$ and $\beta$ are alternating permutations of odd length on disjoint sets (each possibly a singleton) and~$n$ exceeds every entry of both, then the concatenation $\alpha\, n\, \beta$ is again alternating of odd length. Conversely, if~$n$ is the maximum of an alternating permutation $\delta$ of odd length at least $3$, then~$n$ occurs at an even position, and deleting~$n$ splits $\delta$ into two alternating permutations of odd length.

Define the map $f$ recursively. For $n = 1$, send the unique one-vertex tree to the single-block sequence~$1$. For $n \geq 2$, let $T$ be a bicolored decreasing binary tree on~$[n]$, and let $L$ and $R$ denote the left and right subtrees of its root.

\begin{itemize}[itemsep=0pt]
\item If the root has only a right child (so $L$ is empty), let $f(T) = n \mid f(R)$.
\item If the root has only a left child (so $R$ is empty), let $f(T) = f(L) \mid n$.
\item If the root has two children and is colored red, let $f(T) = f(L) \mid n \mid f(R)$.
\item If the root has two children and is colored blue, write $f(L) = L_1 \mid \cdots \mid L_s$ and $f(R) = R_1 \mid \cdots \mid R_t$. Then $f(T) = L_1 \mid \cdots \mid L_{s-1} \mid L_s\, n\, R_1 \mid R_2 \mid \cdots \mid R_t$: the entry~$n$ joins the last block of $f(L)$ to the first block of $f(R)$.
\end{itemize}

In the first three cases,~$n$ forms a singleton block. In the fourth, the merged block $L_s\, n\, R_1$ has odd size because $L_s$ and $R_1$ each do, and it is alternating by the merge-and-split property, since~$n$ is larger than all entries of $L_s$ and $R_1$.

\begin{theorem}
\label{thm:bijection-blocks}
The map $f$ is a bijection from bicolored decreasing binary trees on~$[n]$ to alternating block sequences on~$[n]$.
\end{theorem}

\begin{proof}
To prove the theorem, it suffices to describe the inverse. Let $\alpha = C_1 \mid \cdots \mid C_m$ be an alternating block sequence on~$[n]$, and locate~$n$.

If~$n$ is a singleton block, its position determines the case: at the beginning means the root had only a right child, at the end means only a left child, and in the middle means the root had two children colored red. In each case, removing~$n$ and its block separator leaves one or two alternating block sequences on smaller ground sets, which determine the subtrees by induction.

If~$n$ lies inside a block of size at least $3$, the root had two children colored blue. By the merge-and-split property,~$n$ occurs at an even position within its block, and removing~$n$ splits that block into two alternating permutations $L_s$ and $R_1$ of odd length. Taking $L_s$ together with all preceding blocks gives $f(L)$, and taking $R_1$ together with all subsequent blocks gives $f(R)$, from which the subtrees are recovered by induction.
\end{proof}


\section{A bijection with nonplane trees}
\label{sec:bijection}

A \emph{nonplane decreasing~$1$-$2$ tree} on~$[n]$ is a rooted tree on vertex set~$[n]$ in which every vertex has zero, one, or two children, every vertex has a larger label than its children, and children are not designated left or right. The relabeling $i \mapsto n+1-i$ converts these to nonplane increasing~$1$-$2$ trees, which are counted by the Euler number $E_n$; see, for example, Kuznetsov, Pak, and Postnikov~\cite{kuznetsov:increasing-tree:} or Stanley~\cite[Chapter~5]{stanley:enumerative-com:2}. So nonplane decreasing~$1$-$2$ trees on~$[n]$ are also counted by $E_n$.

Call a vertex with two children \emph{full} and a vertex with one child \emph{unary}. To motivate the bijection, we count the structural data a bicolored decreasing binary tree carries relative to its unordered shape. A unary vertex contributes one binary label, specifying whether its child is left or right. A full vertex contributes two binary labels, one for the left-right designation of its children and one for its own color. Either way, the contribution is one label per child, that is, one label per non-root vertex, matching the $2^{n-1}$ choices of $\varepsilon$ in the theorem below.

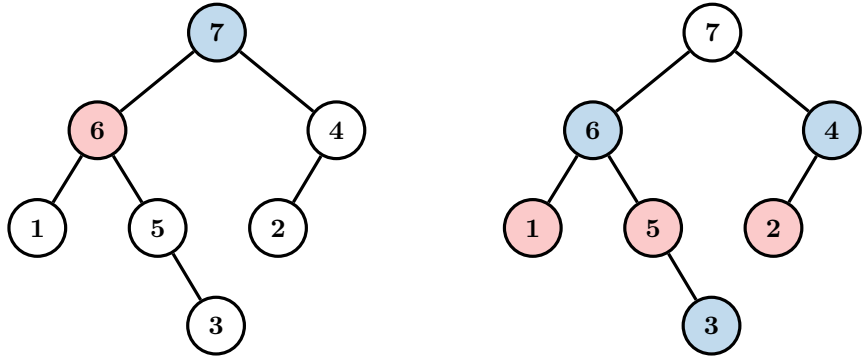
\begin{figure}[t]
\[
\begin{array}{cc@{\qquad}c}
\begin{forest}
my tree
[7, fill=bluedp!40
  [6, fill=reddp!40
    [1]
    [5
      [, no edge, draw=none]
      [3]
    ]
  ]
  [4
    [2]
    [, no edge, draw=none]
  ]
]
\end{forest}
&
&
\begin{forest}
my tree
[7
  [6, fill=bluedp!40
    [1, fill=reddp!40]
    [5, fill=reddp!40
      [, no edge, draw=none]
      [3, fill=bluedp!40]
    ]
  ]
  [4, fill=bluedp!40
    [2, fill=reddp!40]
    [, no edge, draw=none]
  ]
]
\end{forest}
\end{array}
\]
\caption{An illustration of Theorem~\ref{thm:bijection-nonplane}, with red for $0$ and blue for $1$. The bicolored decreasing binary tree on the left has its full vertices ($7$, $6$) colored. On the right is the corresponding pair $(Q, \varepsilon)$: the same vertex set drawn as a nonplane decreasing $1$-$2$ tree, with each non-root vertex carrying a binary label. The label on the the child with smaller label among the children of its parent records the left-right designation (left $= 0$, right $= 1$); the label on the other child records the color of the parent.}
\label{fig:nonplane-bits}
\end{figure}

\begin{theorem}
\label{thm:bijection-nonplane}
There is a bijection between bicolored decreasing binary trees on~$[n]$ and pairs $(Q, \varepsilon)$, where $Q$ is a nonplane decreasing~$1$-$2$ tree on~$[n]$ and $\varepsilon$ assigns a binary label to each non-root vertex of $Q$, as illustrated in Figure~\ref{fig:nonplane-bits}.
\end{theorem}

\begin{proof}
Given $T$, forget the plane embedding and vertex colors to obtain $Q$, and assign the binary label on each non-root vertex as follows. If $v$ is the the child with smaller label among the children of its parent, then $\varepsilon_v = 0$ if $v$ is a left child and $\varepsilon_v = 1$ if $v$ is a right child. If $v$ is not the child of least vertex label of its parent, then the parent of $v$ has a color (because it is full), and $\varepsilon_v$ records this color: $0$ for red and~$1$ for blue.

To recover $T$ from $(Q, \varepsilon)$, we build top-down. At a unary vertex, the child's label restores its direction. At a full vertex with children $v < w$, the label on $v$ determines which child is left and which is right, and the label on $w$ restores the parent's color.
\end{proof}

Since nonplane decreasing~$1$-$2$ trees on~$[n]$ are counted by $E_n$ and each is paired with $2^{n-1}$ choices of $\varepsilon$, we obtain the following.

\begin{corollary}
\label{cor:euler}
For $n \geq 1$, $H_n = 2^{n-1} E_n$.
\end{corollary}

Proposition~\ref{prp:diffeq} gives $H(z) = 1/(1-\tan z)$, and Corollary~\ref{cor:euler} gives $H_n = 2^{n-1} E_n$ for $n \geq 1$; together these give a combinatorial proof of the formula
\[
	\frac{1}{1 - \tan z} = \frac{1 + \sec 2z + \tan 2z}{2}.
\]
Indeed, shifting the index and rescaling $z \mapsto 2z$ in the classical formula $\sum_{n \geq 0} E_n\, z^n/n! = \sec z + \tan z$, we have
\[
	H(z) = 1 + \sum_{n \geq 1} 2^{n-1} E_n\, z^n/n! = 1 + \tfrac{1}{2}\!\left(\sum_{n \geq 0} E_n\, (2z)^n/n! - 1\right) = \frac{1 + \sec 2z + \tan 2z}{2}.
\]


\section{Boolean--Eulerian polynomials}
\label{sec:polynomials}

Recall that the Boolean--Eulerian number $\boel(n,k)$ counts bicolored decreasing binary trees on~$[n]$ with precisely $k-1$ right edges, and that the Boolean--Eulerian polynomial is
\[
	B_n(u) = \sum_{k=1}^{n} \boel(n,k)\, u^{k-1}.
\]
Since reflecting a tree through a vertical axis exchanges left and right edges, we have $\boel(n,k) = \boel(n,n-k+1)$ for all~$n$ and $k$, so the coefficient sequence of $B_n(u)$ is symmetric.

Define the bivariate exponential generating function
\[
	F(z,u) = \sum_{n \geq 0} B_n(u)\, z^n/n!,
\]
noting that $B_0(u) = 1$ because it counts the empty tree.

\begin{proposition}
\label{prp:bivariate-ode}
The generating function $F(z,u)$ satisfies
\[
	\frac{\partial F}{\partial z} = F + u(F-1) + 2u(F-1)^2,
\]
with the initial condition $F(0,u) = 1$.
\end{proposition}

\begin{proof}
Removing the root of a bicolored decreasing binary tree on~$[n]$ yields an ordered pair of bicolored decreasing binary trees on complementary subsets of $[n-1]$. There are four cases, depending on which subtrees are empty and, when both are present, the color of the root. A root with no children (the single-vertex tree) contributes~$1$. A root with only a left child contributes one nonempty left subtree and no new right edge, giving $F - 1$. A root with only a right child contributes one nonempty right subtree and one new right edge, giving $u(F-1)$. A root with two children has a nonempty right subtree (contributing a right edge, hence a factor of $u$), a nonempty left subtree, and a color (red or blue), giving $2u(F-1)^2$. Summing these four contributions yields $1 + (F-1) + u(F-1) + 2u(F-1)^2 = F + u(F-1) + 2u(F-1)^2$.
\end{proof}

Setting $u = 1$ recovers the differential equation $H' = H^2 + (H-1)^2$ from Proposition~\ref{prp:diffeq}.

The first few Boolean--Eulerian polynomials are
\begin{align*}
	B_1(u) &= 1, \\
	B_2(u) &= 1 + u, \\
	B_3(u) &= 1 + 6u + u^2, \\
	B_4(u) &= 1 + 19u + 19u^2 + u^3, \\
	B_5(u) &= 1 + 48u + 158u^2 + 48u^3 + u^4.
\end{align*}
One checks that $B_n(1) = H_n = 2^{n-1} E_n$ agrees with the values $1, 2, 8, 40, 256$.

The bivariate ODE can in principle be solved in closed form for $F(z,u)$ via a Riccati substitution, but we will not need the closed form. Section~\ref{sec:transform} instead uses the structural identity $B_n(u) = g(u)^{n-1} A_n(q(u))$ as the starting point for real-rootedness and interlacing.


\section{The Eulerian transform}
\label{sec:transform}

Recall from the introduction the Eulerian polynomial
\[
	A_n(x) = \sum_{p \in S_n} x^{\des(p)},
\]
whose coefficients are the Eulerian numbers, still not to be confused with the Euler numbers $E_n$. The main result of this section is that $B_n(u)$ is an explicit algebraic transform of $A_n(x)$. Gamma-positivity, real-rootedness, and interlacing all follow, though gamma-positivity also admits a direct combinatorial proof via Foata--Strehl orbits, which we give after the transform.

The transform involves an algebraic change of variable, built from the discriminant $D = \sqrt{u^2 - 6u + 1}$ (choosing the branch with $D(0) = 1$) and the auxiliary functions
\[
	g(u) = \frac{1+u+D}{2}, \qquad h(u) = \frac{1+u-D}{2}, \qquad q(u) = \frac{h(u)}{g(u)}.
\]
Here $g$ and $h$ are the two roots of $t^2 - (1+u)t + 2u = 0$, so $g + h = 1 + u$ and $gh = 2u$. The transform itself is $A_n(x) \mapsto g(u)^{n-1} A_n(q(u))$; the change of variable $q$ and the prefactor $g^{n-1}$ are tied together by the following identities, which we will use throughout the section.

\begin{observation}
\label{obs:gq-identities}
The functions $g$ and $q$ defined above satisfy
\[
\begin{array}{rcccl}
	(q-1)g	&=&	h - g	&=&	-D,		\\
	g(1+q)	&=&	g + h	&=&	1+u,		\\
	g^2 q	&=&	gh		&=&	2u.
\end{array}
\]
\end{observation}

We can now state the transform precisely.

\begin{theorem}
\label{thm:transform}
For all $n \geq 1$,
\[
	B_n(u) = g(u)^{n-1}\, A_n\bigl(q(u)\bigr).
\]
\end{theorem}

\begin{proof}
We show that the right-hand side satisfies the same differential equation and initial condition as $F(z,u)$. The Eulerian polynomials have the exponential generating function
\[
	E(x,z) = \sum_{n \geq 0} A_n(x)\, z^n/n! = \frac{x-1}{x - e^{(x-1)z}},
\]
noting that $A_0(x) = 1$ because it counts the empty permutation (see, for example, Petersen~\cite{petersen:eulerian-number:}), and $E$ satisfies the differential equation
\[
	\frac{\partial E}{\partial z} = xE^2 + (1-x)E.
\]

Define
\[
	G(z,u) = 1 + \frac{1}{g(u)}\bigl(E(q(u),\, g(u)z) - 1\bigr).
\]
Then $G(0,u) = 1$, and for $n \geq 1$ the coefficient of $z^n/n!$ in $G$ is $g(u)^{n-1} A_n(q(u))$, so $G$ is the exponential generating function of the claimed right-hand side. Writing $Y = G - 1$, we have $E(q, gz) = 1 + gY$, so
\[
	\frac{\partial G}{\partial z} = \frac{\partial E}{\partial z}\bigg|_{\substack{x=q \\ z=gz}} = q(1+gY)^2 + (1-q)(1+gY) = 1 + g(1+q)Y + g^2 q\, Y^2.
\]
Applying the identities from Observation~\ref{obs:gq-identities}, this becomes
\[
	\frac{\partial G}{\partial z} = 1 + (1+u)Y + 2uY^2 = G + u(G-1) + 2u(G-1)^2.
\]
Since $G$ satisfies the same differential equation and initial condition as $F(z,u)$ from Proposition~\ref{prp:bivariate-ode}, formal-power-series uniqueness gives $G = F$, and comparing coefficients proves the theorem.
\end{proof}

%
%

We turn to gamma-positivity. Any polynomial of degree $n-1$ with symmetric coefficient sequence can be uniquely expanded in the basis $\{x^j (1+x)^{n-1-2j}\}_{j \geq 0}$, and the coefficients in this expansion are called the \emph{gamma-coefficients}. The classical gamma expansion of the Eulerian polynomial (see Foata and Sch\"utzenberger~\cite{foata:theorie-geometr:} or Petersen~\cite{petersen:eulerian-number:}) is
\[
	A_n(x) = \sum_{j \geq 0} \gamma_{n,j}\, x^j\,(1+x)^{n-1-2j},
\]
where the~$\gamma_{n,j}$ count permutations of $[n]$ with $j$ peaks and no double descents; in particular, the~$\gamma_{n,j}$ are nonnegative integers, so the Eulerian polynomial~$A_n(x)$ is \emph{gamma-positive}.

The standard combinatorial proof of this fact, due to Foata and Strehl~\cite{foata:rearrangements-:}, groups permutations into orbits under a valley-hopping action whose canonical representatives are the permutations counted by~$\gamma_{n,j}$. Call a decreasing binary tree on $[n]$ \emph{left-normal} if every unary vertex has its child on the left, and write $\mathcal{L}_{n,j}$ for the set of left-normal decreasing binary trees on $[n]$ with precisely $j$ full vertices. Under the bijection between permutations and decreasing binary trees, no-double-descent permutations correspond precisely to left-normal trees, with peaks corresponding to full vertices, so the Foata--Strehl theorem gives $|\mathcal{L}_{n,j}| = \gamma_{n,j}$.

In the decreasing-binary-tree model the valley-hopping action becomes especially transparent, and the labeled refinement falls out with the same effort:

\begin{proposition}
\label{prp:foata-strehl}
For all $n \geq 1$,
\[
	B_n(u) = \sum_{j \geq 0} 2^j\, \gamma_{n,j}\, u^j (1+u)^{n-1-2j}.
\]
\end{proposition}

\begin{proof}
We construct a bijection between bicolored decreasing binary trees $T$ on $[n]$ and triples $(R, c, S)$, where $R$ is a left-normal decreasing binary tree on $[n]$, $c$ is a coloring of the full vertices of~$R$ by red and blue, and $S$ is a subset of the unary vertices of~$R$.

Given a triple $(R, c, S)$, build $T$ as follows. Since $R$ is left-normal, every unary vertex of $R$ has its child on the left; flip the children of the unary vertices in $S$ to the right, leaving the children of the unary vertices outside $S$ on the left and leaving the full vertices and their colors untouched. This operation preserves the decreasing labeling and the full/unary classification of each vertex, so the result $T$ is a bicolored decreasing binary tree on $[n]$ with the same full vertices and the same colors as $R$.

Conversely, given a bicolored decreasing binary tree $T$ on $[n]$, we can recover the triple as follows. Take $R$ to be the left-normal tree obtained from $T$ by flipping every right-unary vertex to the left, take $c$ to be the colors of the full vertices, and take $S$ to be the set of unary vertices that were right-children in $T$. The forward and backward constructions are inverse to each other.

Suppose $R \in \mathcal{L}_{n,j}$. Counting edges of $R$, the $j$ full vertices contribute $2j$ edges and the unary vertices contribute one each, accounting for all $n-1$ edges, so $R$ has $n - 1 - 2j$ unary vertices. The right edges of $R$ are precisely the right children of its full vertices, contributing $j$ right edges, and each unary vertex flipped to the right adds one more. The trees mapping to $R$ therefore contribute
\[
	2^j \sum_{S \subseteq \mathrm{Unary}(R)} u^{j + |S|} = 2^j\, u^j\, (1+u)^{n-1-2j}
\]
to $B_n(u)$, where the factor $2^j$ counts the colorings $c$. Summing over $j$ and over $R \in \mathcal{L}_{n,j}$ proves the proposition, using the fact that $|\mathcal{L}_{n,j}| = \gamma_{n,j}$.
\end{proof}

Since the coefficients $2^j \gamma_{n,j}$ are nonnegative, we obtain the following.

\begin{corollary}
\label{cor:gamma}
The Boolean--Eulerian polynomials are gamma-positive.
\end{corollary}

%
%

The same formula also follows algebraically: substituting $x = q(u)$ in the Eulerian gamma expansion, multiplying by $g(u)^{n-1}$, and applying the identities from Observation~\ref{obs:gq-identities} converts $g^2 q$ to $2u$ and $g(1+q)$ to $1+u$.

%
%

Since the gamma-coefficients $2^j \gamma_{n,j}$ of the Boolean--Eulerian polynomials are nonnegative, the symmetry and unimodality of the coefficient sequence of $B_n(u)$ follow immediately. Real-rootedness does not follow from gamma-positivity in general, so we prove it separately. The key tool is monotonicity of $q$ on the negative reals.

\begin{corollary}
\label{cor:q-properties}
The function $q\colon (-\infty, 0) \to (-\infty, 0)$ is continuous and strictly increasing, and takes on all negative real values. In particular, $q$ is a bijection.
\end{corollary}

\begin{proof}
We have $u^2 - 6u + 1 > 1$ for $u < 0$, so the discriminant $D = \sqrt{u^2 - 6u + 1}$ is real and positive for these values. Moreover $D^2 - (1+u)^2 = -8u > 0$, so $D > |1+u|$. Thus
\[
	q(u) = \frac{1+u-D}{1+u+D} < 0.
\]
Differentiating $q$ gives
\[
	q'(u) = \frac{(1-D')(1+u+D) - (1+u-D)(1+D')}{(1+u+D)^2} = \frac{2\bigl(D - (1+u)D'\bigr)}{(1+u+D)^2}.
\]
Implicit differentiation of $D^2 = u^2 - 6u + 1$ gives $D' = (u-3)/D$. Substituting this and using~${D^2 = u^2 - 6u + 1}$ simplifies the numerator to $8(1-u)/D$, and so
\[
	q'(u) = \frac{8(1-u)}{D(1+u+D)^2} > 0
\]
for $u < 0$.
Thus $q$ is strictly increasing on $(-\infty, 0)$. Since $q(u) \to -\infty$ as $u \to -\infty$ and $q(u) \to 0$ as~${u \to 0^-}$, the function $q$ maps $(-\infty, 0)$ bijectively onto $(-\infty, 0)$.
\end{proof}

We can now transport real-rootedness from $A_n$ to $B_n$.

\begin{theorem}
\label{thm:real-roots}
For all $n \geq 2$, the Boolean--Eulerian polynomial $B_n(u)$ has $n-1$ real, simple, negative roots.
\end{theorem}

\begin{proof}
The Eulerian polynomial $A_n(x)$ has degree $n-1$ with $n-1$ real, simple, negative roots; this goes back to Frobenius~\cite{frobenius:uber-die-bernou:}, and can also be found in Brenti~\cite{brenti:unimodal-log-co:} and Petersen~\cite{petersen:eulerian-number:}. By Theorem~\ref{thm:transform}, $B_n(u) = g(u)^{n-1} A_n(q(u))$. For $u < 0$, the proof of Corollary~\ref{cor:q-properties} shows that $g(u) = (1+u+D)/2 > 0$, so the roots of $B_n$ in $(-\infty, 0)$ are the preimages under $q$ of the roots of $A_n$. Since $q$ is a bijection from $(-\infty, 0)$ to $(-\infty, 0)$ by Corollary~\ref{cor:q-properties}, each of the $n-1$ negative roots of $A_n$ has a unique negative preimage. Since $\deg B_n = n-1$, these account for all roots.
\end{proof}

As real-rootedness and positive coefficients give log-concavity (via Newton's inequalities), we have the following.

\begin{corollary}
\label{cor:logconcave}
For fixed~$n$, the sequence $\boel(n,1), \boel(n,2), \ldots, \boel(n,n)$ is log-concave.
\end{corollary}

Beyond root structure for a single $n$, the transform also yields a relation between consecutive polynomials. The Eulerian polynomials $\{A_n(x)\}_{n \geq 2}$ have interlacing roots (see Savage and Visontai~\cite{savage:the-s-eulerian-:}), and since $q^{-1}$ is order-preserving by Corollary~\ref{cor:q-properties}, the interlacing is inherited by the images under~$q^{-1}$, which by Theorem~\ref{thm:transform} are the roots of $B_n$ and $B_{n-1}$ in $(-\infty, 0)$. Putting this together yields:

\begin{corollary}
\label{cor:interlacing}
For $n \geq 3$, the roots of $B_{n-1}(u)$ strictly interlace the roots of $B_n(u)$: between any two consecutive roots of $B_n(u)$, there is precisely one root of $B_{n-1}(u)$.
\end{corollary}


\section{Combinatorial unimodality}
\label{sec:unimodal}

The gamma-positivity established in Corollary~\ref{cor:gamma} already implies that the coefficient sequence of~$B_n(u)$ is unimodal (as does the log-concavity established in Corollary~\ref{cor:logconcave}). We give an independent combinatorial proof, because it illustrates how the injection from the Boolean--Narayana setting~\cite{bona:boolean--naraya:} extends to labeled trees.

\begin{theorem}
\label{thm:unimodal}
For all fixed~$n$, the sequence $\boel(n,1), \boel(n,2), \ldots, \boel(n,n)$ is unimodal.
\end{theorem}

\begin{proof}
Since the sequence is symmetric, it suffices to show that $\boel(n,k) \leq \boel(n,k+1)$ for $k \leq (n-1)/2$. Write $\mathcal{E}(n,k)$ for the set of bicolored decreasing binary trees on~$[n]$ with $k-1$ right edges, so $|\mathcal{E}(n,k)| = \boel(n,k)$.

B\'ona~\cite{bona:boolean--naraya:} constructs an injection proving the corresponding inequality for Boolean--Narayana numbers. The strategy is to identify a canonical initial piece of the tree in which left edges outnumber right edges by exactly one, and then swap the left-right designations of the unary vertices in that piece. The swap converts one such left edge into a right edge, raising the right-edge count by one, and the canonical piece is recoverable from the image, so the swap is invertible.

We recall the construction. Order the vertices of a tree by decreasing depth, and within each depth from left to right, ending at the root. Given a tree with $k-1$ right edges and $k \leq (n-1)/2$, let $T_i$ denote the subforest induced by the first~$i$ vertices in this ordering, and write $\ell(T_i)$ and $r(T_i)$ for its numbers of left and right edges. Then $\ell(T_2) - r(T_2) \leq 1$, while $\ell(T_n) - r(T_n) = (n-k) - (k-1) \geq 2$. When a vertex is added to the subforest, it contributes either no edge, a single left or right edge to its only child, or one left and one right edge to its two children. Hence $\ell(T_i) - r(T_i)$ changes by at most~$1$ at each step, so there is a smallest index~$i$ for which $\ell(T_i) - r(T_i) = 1$. Swap left and right subtrees at every unary vertex within~$T_i$, leaving the rest of the tree unchanged. Inside~$T_i$ this turns~${\ell(T_i)}$ left edges into right edges and vice versa, so the new subforest has~${\ell(T_i)}$ right edges and~${r(T_i)}$ left edges, a net gain of one right edge over the original~$T_i$; edges outside~$T_i$ are unchanged, so the whole tree gains exactly one right edge and now has~$k$. To invert the map, locate the smallest index~$j$ such that the modified subforest has one more right edge than left edge: this~$j$ equals our~$i$, because the swap exchanged the values of~$\ell$ and~$r$ on $T_i$ while preserving them outside. So $T_i$ is recovered from the image, and reversing the swap recovers the original tree.

This construction modifies only edge directions. It preserves the underlying parent-child relationships, the vertex labels, and which vertices have two children; in particular, it preserves both the decreasing labeling and the colors. Applying it to trees in $\mathcal{E}(n,k)$ therefore gives an injection into~${\mathcal{E}(n,k+1)}$ for all $k \leq (n-1)/2$.
\end{proof}


\section{Further directions}
\label{sec:further}

In both the unlabeled setting of binary plane trees, Boolean--Catalan, and Boolean--Narayana numbers~\cite{bona:stack-sorting-p:, bona:boolean--naraya:}, and the labeled setting of this paper, coloring the full vertices red or blue produces a family whose generating polynomials are an explicit transform of the uncolored originals, and the nice enumerative properties of the Narayana and Eulerian polynomials transfer through these transforms. Two questions suggest themselves.

The first asks whether this pattern extends to other families. Are there other families of trees in bijection with permutations, where right edges correspond to descents, and where a two-coloring of full vertices produces a similarly well-behaved refinement? More broadly, what is the right framework in which transforms like $B_n(u) = g(u)^{n-1} A_n(q(u))$ should be expected? B\'ona raised a version of this question in~\cite{bona:boolean--naraya:}, and Theorem~\ref{thm:bijection-nonplane}, which connects bicolored decreasing binary trees to nonplane decreasing~$1$-$2$ trees and the Euler numbers, suggests that the interplay between plane embeddings, labels, and colorings might have more to offer.

The second concerns vertical log-concavity. The Boolean--Narayana paper~\cite{bona:boolean--naraya:} establishes that for fixed $k$, the sequence $\mathrm{BoNa}(n,k)$ for $n \geq k+1$ is log-concave. For the Boolean--Eulerian numbers, even the classical Eulerian analogue appears open: we are not aware of a proof that the column sequence $\bigl(\langle {n \atop k} \rangle\bigr)_{n \geq k+1}$ of Eulerian numbers is log-concave for fixed $k$. (Brenti's stronger conjecture~\cite{brenti:the-application:}, that the Eulerian triangle is totally positive, also remains open.) Subject to these qualifications, we conjecture the following.

\begin{conjecture}
\label{conj:vertical}
For fixed $k \geq 2$, the sequence $\boel(n,k)$ for $n \geq k$ is log-concave.
\end{conjecture}

%
%
%
%
%


\begin{thebibliography}{10}

\bibitem{bona:boolean--naraya:}
{\sc B\'{o}na, M.}
\newblock {B}oolean--{N}arayana numbers.
\newblock arXiv:2602.11355 [math.CO].

\bibitem{bona:stack-sorting-p:}
{\sc B\'{o}na, M.}
\newblock Stack-sorting preimages and $0$-$1$-trees.
\newblock arXiv:2505.18295 [math.CO].

\bibitem{brenti:unimodal-log-co:}
{\sc Brenti, F.}
\newblock Unimodal, log-concave and {P}\'olya frequency sequences in
  combinatorics.
\newblock {\em Mem. Amer. Math. Soc. 81}, 413 (1989), viii+106.

\bibitem{brenti:the-application:}
{\sc Brenti, F.}
\newblock The applications of total positivity to combinatorics, and
  conversely.
\newblock In {\em Total positivity and its applications ({J}aca, 1994)},
  vol.~359 of {\em Math. Appl.} Kluwer Acad. Publ., Dordrecht, 1996,
  pp.~451--473.

\bibitem{foata:theorie-geometr:}
{\sc Foata, D., and Sch\"utzenberger, M.-P.}
\newblock {\em Th\'eorie g\'eom\'etrique des polyn\^omes eul\'eriens},
  vol.~138 of {\em Lecture Notes in Mathematics}.
\newblock Springer-Verlag, Berlin-New York, 1970.

\bibitem{foata:rearrangements-:}
{\sc Foata, D., and Strehl, V.}
\newblock Rearrangements of the symmetric group and enumerative properties of
  the tangent and secant numbers.
\newblock {\em Math. Z. 137\/} (1974), 257--264.

\bibitem{frobenius:uber-die-bernou:}
{\sc Frobenius, G.}
\newblock \"{U}ber die {Bernoullischen} {Zahlen} und die {Eulerschen}
  {Polynome}.
\newblock {\em Berl. Ber. 1910\/} (1910), 809--847.

\bibitem{kuznetsov:increasing-tree:}
{\sc Kuznetsov, A.~G., Pak, I.~M., and Postnikov, A.~E.}
\newblock Increasing trees and alternating permutations.
\newblock {\em Uspekhi Mat. Nauk 49}, 6(300) (1994), 79--110.

\bibitem{petersen:eulerian-number:}
{\sc Petersen, T.~K.}
\newblock {\em Eulerian numbers}.
\newblock Birkh\"auser Advanced Texts: Basler Lehrb\"ucher.
  Birkh\"auser/Springer, New York, 2015.

\bibitem{savage:the-s-eulerian-:}
{\sc Savage, C.~D., and Visontai, M.}
\newblock The {$\bold{s}$}-{E}ulerian polynomials have only real roots.
\newblock {\em Trans. Amer. Math. Soc. 367}, 2 (2015), 1441--1466.

\bibitem{stanley:enumerative-com:2}
{\sc Stanley, R.~P.}
\newblock {\em Enumerative Combinatorics, Vol. 2}, vol.~62 of {\em Cambridge
  Stud. in Advanced Math.}
\newblock Cambridge University Press, Cambridge, England, 1999.

\end{thebibliography}

%
%
%
%
%

\end{document}